\documentclass[reqno,oneside,12pt]{amsart}

\usepackage[T1]{fontenc}
\usepackage{times,mathptm}
\usepackage{amssymb,epsfig,verbatim,xypic}
\def\vs{\vspace{0.1cm}}

\theoremstyle{plain}

\newtheorem{thm}{Theorem}[section]

\newtheorem{cor}[thm]{Corollary}
\newtheorem{pro}[thm]{Proposition}
\newtheorem{lem}[thm]{Lemma}
\newtheorem{proposition-principale}[thm]{Proposition principale}
\newtheorem{thm-principal}{Th\'eor\`eme principal}[section]

\theoremstyle{definition}
\newtheorem{defi}[thm]{Definition}

\newtheorem{rem}[thm]{Remark}

\newenvironment{congb}
{{\vs \noindent \bf Geometric Bogomolov Conjecture.--$\,$}\it}{\vs}

\newenvironment{thm-A}
{{\vs \noindent \bf Theorem A.$\,$}\it}{\vs}

\newenvironment{thm-B}
{{\vs \noindent \bf Theorem B.$\,$}\it}{\vs}

\newenvironment{thm-BB}
{{\vs \noindent \bf Theorem B'.$\,$}\it}{\vs}

\def\C{\mathbf{C}}
\def\R{\mathbf{R}}
\def\Q{\mathbf{Q}}
\def\Z{\mathbf{Z}}

\def\bk{\mathbf{k}}

\def\J{{\textsc{j}}}

\newcommand{\sA}{{\mathcal A}}

\newcommand{\sF}{{\mathcal F}}

\newcommand{\sL}{{\mathcal L}}

\newcommand{\sX}{{\mathcal X}}

\def\ess{\mathrm{ess}}

\def\End{{\sf{End}}}
\def\GL{{\sf{GL}}}
\def\SL{{\sf{SL}}}

\def\tr{{\sf{tr}}}

\def\dist{{\sf{dist}}}

\def\fieldchar{{\rm{char}\,}}

\addtocounter{section}{0}             
\numberwithin{equation}{section}       

\begin{document}

\setlength{\baselineskip}{0.53cm}        
%
%
\title[]
{The geometric Bogomolov conjecture}
\date{2018}

\author{Serge Cantat}
\address{Serge Cantat, IRMAR, Campus de Beaulieu,
b\^atiments 22-23
263 avenue du G\'en\'eral Leclerc, CS 74205
35042  RENNES C\'edex}
\email{serge.cantat@univ-rennes1.fr}

\author{Ziyang Gao}
\address{Ziyang Gao, CNRS, IMJ-PRG, 4 place de Jussieu, 75000 Paris, France; Department of Mathematics, Princeton University, Princeton, NJ 08544, USA}
\email{ziyang.gao@imj-prg.fr}

\author{Philipp Habegger}
\address{Philipp Habegger, Department of Mathematics and Computer Science, University of Basel, Spiegelgasse 1,4051 Basel, Switzerland}
\email{philipp.habegger@unibas.ch}

\author{Junyi Xie}
\address{Junyi Xie, IRMAR, Campus de Beaulieu,
b\^atiments 22-23
263 avenue du G\'en\'eral Leclerc, CS 74205
35042  RENNES C\'edex}
\email{junyi.xie@univ-rennes1.fr}

\thanks{The last-named author is partially supported by project ``Fatou'' ANR-17-CE40-0002-01, the first-named author by the french academy of sciences (fondation del Duca). The second- and third-named authors thanks the University of Rennes 1 for its hospitality, and the foundation del Duca for financial support. }

%
%

\maketitle
 

%
%
%
%

\begin{abstract}
We prove the geometric Bogomolov conjecture over a function field of characteristic zero.
\end{abstract}

%


\section{Introduction}

\subsection{The geometric Bogomolov conjecture}

\subsubsection{Abelian varieties and heights}

Let $\bk$ be an algebraically closed field.
Let $B$ be an irreducible normal projective variety over $\bk$ of dimension 
$d_B\geq 1$. Let $K:=\bk(B)$ be the function field of $B$.
Let $A$ be an abelian variety  defined over $K$  of dimension $g$. Fix an ample line bundle 
$M$ on $B$, and a symmetric ample line bundle $L$ on $A$.

Denote by $\hat{h}:A(\overline{K})\to [0,+\infty)$ the canonical
  height on $A$ with respect to $L$ and $M$ where $\overline{K}$ is an
  algebraic closure of $K$ (see Section~\ref{par:canonical height}).
For any irreducible subvariety $X$ of  $A_{\overline{K}}$ and any $\epsilon> 0$, we set 
\begin{equation}
X_{\epsilon}:=\{x\in X(\overline{K})|\,\, \hat{h}(x)<\epsilon\}.
\end{equation}

Set $A_{\overline{K}}= A\otimes_K\overline{K}$, and denote by $(A^{\overline{K}/\bk}, \tr)$ the $\overline{K}/\bk$-trace of  $A_{\overline{K}}$: it
is the final object of the category of pairs $(C,f)$, where $C$ is an abelian variety over $\bk$ and $f$ is 
a morphism from $C\otimes_{\bk}\overline{K}$ to  $A_{\overline{K}}$ (see~\cite{Lang1983}).
If $\fieldchar\bk=0$, $\tr$ is a closed immersion and $A^{\overline{K}/\bk}\otimes_{\bk}\overline{K}$ can be naturally viewed 
as an abelian subvariety of $A_{\overline{K}}$. By definition, a {\bf{torsion coset}} of $A$ is a translate $a+C$ of an abelian 
subvariety $C\subset A$ by a torsion point $a$. An irreducible subvariety $X$ of $A_{\overline{K}}$ is said to be {\bf{special}} if  
\begin{equation}
X=\tr(Y{\otimes_{\bk}\overline{K}})+T
\end{equation}
for some torsion coset $T$ of $A_{\overline{K}}$ and  some subvariety $Y$ of $A^{\overline{K}/\bk}$.
When $X$ is special,  $X_\epsilon$ is Zariski dense in $X$ for all $\epsilon >0$ (\cite[Theorem 5.4, Chapter 6]{Lang1983}).

\subsubsection{Bogomolov conjecture}

 The following conjecture was proposed by Yamaki \cite[Conjecture 0.3]{Yamaki2013}, but particular
 instances of it were studied earlier by Gubler in~\cite{Gubler2007}. It is an analog over function fields of the Bogomolov conjecture which 
 was proved by Ullmo \cite{Ullmo1998} and Zhang \cite{Zhang1998}.

\begin{congb} Let $X$ be an irreducible subvariety of $A_{\overline{K}}$. 
If $X$ is not special there exists $\epsilon>0$ such that $X_{\epsilon}$ is not Zariski dense in $X$.
\end{congb}

\smallskip

The aim of this paper is to prove the geometric Bogomolov conjecture over a function field of characteristic zero.

\smallskip

\begin{thm-A} Assume that $\bk$ is an algebraically closed field of characteristic $0$. 
Let $X$ be an irreducible subvariety of $A_{\overline{K}}$. If $X$ is not special then there exists $\epsilon>0$ such 
that $X_{\epsilon}$ is not Zariski dense in $X.$
\end{thm-A}

\vspace*{-0.13cm}

\subsubsection{Historical note}

Gubler proved the geometric Bogomolov conjecture in~\cite{Gubler2007} when $A$ is totally degenerate at some place of $K$.
When $\dim B=1$ and $X\subset A$ is a curve in its Jacobian, Yamaki proved it for nonhyperelliptic curves of genus $3$ in \cite{Yamaki2002} and for any hyperelliptic curve in \cite{Yamaki2008}. 
If moreover $\fieldchar \bk=0$, Faber \cite{Faber2009} proved it if $X$ is a curve of genus at most $4$ and Cinkir \cite{Cinkir2011} 
covered the case of arbitrary genus.
Later on Yamaki proved the cases $({\rm co})\dim X= 1$ \cite{Yamaki2017} and $\dim(A^{\overline{K}/{\bk}})\geq \dim(A)-5$ \cite{Yamaki2017a};
in \cite{Yamaki2016a}, he   reduced the conjecture to the case of  abelian varieties with trivial 
$\overline{K}/k$-trace and good reduction everywhere. 
In \cite{Habegger2013}, the third-named author gave a new proof of this conjecture in characteristic $0$ when  $A$ is the power of an elliptic curve and $\dim B =1$, introducing the original idea of considering the Betti map and its monodromy. 
Recently, the second and the third-named authors \cite{Gao2018} proved the  conjecture in the case $\fieldchar\bk=0$ and $\dim B=1$.

\subsection{An overview of the proof of Theorem~A}

\subsubsection{Notation}\label{par:Notation}

From now on, the algebraically closed field $\bk$ has characteristic $0$. 
There exists an algebraically closed subfield 
$\bk'$ of $\bk$ such that $B$, $A$, $X$, $M$ and $L$ are defined over 
 $\bk'$ and the transcendental degree of $\bk'$ over $\overline{\Q}$ is finite. 
In particular, $\bk'$ can be embedded in the complex field $\C$. Thus, 
\emph{in the rest of the paper, we assume $\bk=\C$} and we denote by $K$ the function field $\C(B)$.

Let $\pi:\sA\to B$ be an irreducible projective scheme over $B$ whose generic fiber is isomorphic to $A$. 
We may assume that $\sA$ is normal, and we fix an ample line bundle $\sL$ on $\sA$ such that $\sL|_A=L$.
For $b\in B$, we set $\sA_b:=\pi^{-1}(b)$. 
We denote by $e: B\dasharrow \sA$ the zero section and
by $[n]$ the multiplication by $n$ on $A$; it defines a rational mapping $\sA\dasharrow \sA$.

We may assume that $M$ is very ample, and we fix an embedding of $B$ in a projective space such that
the restriction of ${\mathcal{O}}(1)$ to $B$ coincides with $M$. The restriction of the Fubini-Study form
to $B$ is a K\"ahler form $\nu$.

Fix  a Zariski dense open subset $B^o$ of $B$ such that $B^o$ is smooth and $\pi|_{\pi^{-1}(B^o)}$ is smooth; then,
set $\sA^o:=\pi^{-1}(B^o)$. 

Let $X$ be a geometrically irreducible subvariety of $A$ such that $X_\epsilon$ is Zariski dense in $X$ for every $\epsilon >0$. We denote by $\sX$ 
 its Zariski closure in $\sA$, by $\sX^o$ its Zariski closure in $\sA^o$, and by $\sX^{o,reg}$ the regular locus of $\sX^o$. Our
 goal is to show that $X$ is special. 

\vspace*{-0.1cm}

\subsubsection{The main ingredients}

One of the main ideas of this paper is to consider the Betti foliation (see Section \ref{subsecbettimap}). It is
a smooth foliation of $\sA^o$ by holomorphic leaves, which is transverse to $\pi$. 
Every torsion point of $A$ gives local sections of  $\pi|_{\pi^{-1}(B^o)}$: these sections are local leaves of the Betti foliation, and this
property characterizes it. 

To prove Theorem~A, the {\bf first step} is to show that $\sX^o$ is 
invariant under the foliation when small points are dense in $X$. 
In other words, at every smooth point $x\in\sX^o$, the tangent space to the Betti foliation is contained in $T_x{\sX^o}$.
For this, we introduce a semi-positive closed $(1,1)$-form $\omega$ on $\sA^o$ which is canonically associated to $L$ and vanishes along the foliation.  
An inequality of Gubler implies that the canonical height $\hat{h}(X)$ of $X$ is $0$ when small points are dense in $X$; 
Theorem~B asserts that the  condition $\hat{h}(X)=0$ translates into 
\begin{equation}
\int_{\sX^{o}}\omega^{\dim X+1}\wedge (\pi^*\nu)^{m-1}=0
\end{equation} 
where $\nu$ is any K\"ahler form on the base $B^o$. 
From the construction of $\omega$, we deduce that $X$ is invariant
under the Betti foliation.

 The first step implies that the fibers of $\pi|_{\sX^o}$ are invariant 
under the action of the holonomy of the Betti foliation; the {\bf second step} shows that a subvariety 
of a fiber $\sA_b$ which is invariant under the holonomy is  the sum of a torsion coset  and a subset of $A^{\overline{K}/\bk}$. 
The conclusion easily follows from these two main steps. 
The second step already appeared in \cite{Habegger2013} and \cite{Gao2018}, but here, we make use of a more efficient dynamical argument
which may be derived from a result of Muchnik and is independent of the  Pila-Zannier's counting strategy.

\subsection{Acknowledgement} 


 The authors thank Pascal Autissier and Walter Gubler for providing comments and references. 

\section{The Betti form}\label{sectionbetti}

In this section, we define a foliation, and a closed $(1,1)$-form on $\sA^o$ which is 
naturally associated to the line bundle $L$.

\subsection{The local Betti maps}\label{subsecbettimap}

Let $b$ be a  point of $B^o$, and $U\subseteq B^o(\C)$ be a connected and simply connected open neighbourhood  of $b$ in 
the euclidean topology. Fix a basis of 
$H_1(\sA_b;\Z)$ and extend it by continuity to all fibers above $U$.
There is a natural real analytic diffeomorphism $\phi_U:\pi^{-1}(U)\to U\times\R^{2g}/\Z^{2g}$ 
such that 
\begin{enumerate}
\item $\pi_1\circ\phi_U=\pi$ where $\pi_1:U\times\R^{2g}/\Z^{2g}\to U$ is the projection to the first factor;
\item  for every $b\in U$, the map $\phi_U|_{\sA_{b}}:\sA_{b}\to \pi_1^{-1}(b)$ is an isomorphism of real Lie groups
that maps the basis of  $H_1(\sA_{b};\Z)$ onto the canonical basis of $\Z^{2g}$. 
\end{enumerate} 
For $b$ in $U$, denote by $i_b: \R^{2g}/\Z^{2g}\to U\times\R^{2g}/\Z^{2g}$ the inclusion $y\mapsto (b,y)$.
The {\bf{Betti map}} is  the $C^{\infty}$-projection $\beta_U^b\colon \pi^{-1}(U)\to \sA_b$ defined by
\begin{equation}
\beta_U^b:=(\phi_U|_{\sA_b})^{-1}\circ i_b\circ\pi_2\circ\phi_U 
\end{equation}
where $\pi_2:U\times\R^{2g}/\Z^{2g}\to \R^{2g}/\Z^{2g}$ is the projection to the second factor. 
Changing the  basis of $H_1(\sA_b;\Z)$, we obtain another trivialization $\phi'_U$ 
that is given by post-composing  $\phi_U$ with a
constant linear transformation 
\begin{equation}
(b,z)\in U\times\R^{2g}/\Z^{2g}\mapsto (b, h(z))
\end{equation}
for some element $h$ of the group $\GL_{2g}(\Z)$; thus, $\beta_U^b$ does not depend on $\phi_U$. 

Note that $\beta_U^b$ is the identity on $\sA_b$. 
In general, $\beta_U^b$ is not holomorphic. However, for every $p\in \sA_b$, $(\beta_U^{b})^{-1}(p)$ 
is a complex submanifold of $\sA^o$. (For instance, every section of $\pi|_{\pi^{-1}U}$ which is given 
by a torsion point provides a fiber of $\beta_U^b$, and continous limits of holomorphic sections are holomorphic.)

\subsection{The Betti foliation}

The local Betti maps determine a natural foliation $\sF$ on $\sA^o$: for every point $p$,  the local leaf $\sF_{U,p}$ 
through $p$ is the fiber $(\beta^{\pi(p)}_U)^{-1}(p)$. 
We call $\sF$ the {\bf{Betti foliation}}.
The leaves of $\sF$ are holomorphic, in the following sense: for every $p\in \sA^o$, the local leaf $\sF_{U,p}$ 
is a complex submanifold of $\pi^{-1}(U)\subset \sA^o$.  But a global leaf $\sF_p$ can be dense in $\sA^o$ for the
euclidean topology. Moreover, $\sF$ is everywhere transverse to the fibers of $\pi$, 
and $\pi|_{\sF_p}\colon \sF_p\to B^o$ is a regular holomorphic covering  for every point $p$ 
(it may have finite or infinite degree, and this may depend on $p$).

\begin{rem}\label{rem:torsion-algebraic-leaves}
The foliation $\sF$ is characterized as follows.
Let $q$ be a torsion point of $\sA_b$; it determines a multisection of the fibration $\pi$, obtained by analytic
continuation of $q$ as a torsion point in nearby fibers of $\pi$. This
multisection coincides with the leaf  $\sF_q$. There is a unique foliation of $\sA^o$ 
which is everywhere transverse to $\pi$ and whose set of leaves contains all those multisections. 
\end{rem}

\begin{rem}\label{rem:[n]-fol}
One can also think about $\sF$ dynamically. The endomorphism $[n]$ determines
 a rational transformation  of the model $\sA$ and induces a regular transformation of $\sA^o$. It preserves $\sF$, 
mapping leaves to leaves. Preperiodic leaves correspond to preperiodic points of $[n]$ in the fiber $\sA_b$; they are 
exactly the leaves given by the torsion points of $A$. 
\end{rem}

\begin{rem}\label{remtrivial} Assume that the family $\pi:\sA^o\to B^o$ is trivial, i.e. $\sA^o=B^o\times A_{\C}$ where $A_{\C}$ is an abelian variety over $\C$ and $\pi$ is the first projection. Then, the leaves of $\sF$ are exactly the fibers of the second projection. \end{rem}

\subsection{The Betti form}\label{subsec:bettiform}
The Betti form is introduced by Mok in \cite[pp.~374]{Mok11Form} to study the Mordell-Weil group over function fields. 
We hereby sketch the construction of this $(1,1)$-form. 
For $b\in B^o$, there exists a unique smooth $(1,1)$-form $\omega_b\in c_1(\sL|_{\sA_b})$ on $\sA_b$ 
which is invariant under translations.  
If we write $\sA_b=\C^g/\Lambda$ and denote by $z_1,\dots,z_{g}$ the standard coordinates of $\C^g$, 
then
\begin{equation}
\omega_b=\sum_{1\leq i, j\leq g}a_{i,j}dz_i\wedge d\bar{z_j}
\end{equation}
for some complex numbers $a_{i,j}$. This form $\omega_b$ is positive, because $\sL|_{\sA_b}$ is ample. 

Now, we define a smooth $2$-form $\omega$ on $\sA^o$. Let $p$ be a point of $\sA^o$. 
First, define $P_p\colon T_p\sA^o\to T_p\sA_{\pi(p)}$ to be  the projection onto the first factor in 
\begin{equation}
T_p\sA^o= T_p\sA_{\pi(p)}\oplus T_p\sF. 
\end{equation}
Since the tangent spaces $T_p\sF$ and $T_p\sA_{\pi(p)}$ are complex subspaces of $T_p\sA^o$, the map $P_p$ 
is a complex linear map. Then, for  $v_1$ and $v_2\in T_p{\sA^o}$ we set
\begin{equation}
\omega(v_1,v_2):=\omega_{\pi(p)}(P_p(v_1),P_p(v_2)).
\end{equation}
We call $\omega$ the {\bf{Betti form}}. By construction, $\omega|_{\sA_b}=\omega_b$ for every $b$.  Since $\omega_b$ is of type $(1,1)$ and  $P_p$ is $\C$-linear, $\omega$ is an antisymmetric form of type $(1,1)$. Since $\omega_b$ is positive, $\omega$ is semi-positive.

Let $U$ and $\phi_U$  be as in Section~\ref{subsecbettimap}. Let $y_i$, $i=1,\dots, 2g$, denote the standard coordinates of $\R^{2g}$.
Then there are real numbers $b_{i,j}$ such that 
\begin{equation}
(\phi_U^{-1})^{*}\omega = \sum_{1\leq i<j\leq 2g}b_{i,j}dy_i\wedge dy_j.
\end{equation}
It follows that $d((\phi_U^{-1})^{*}\omega)=0$ and that $\omega$ is closed. Moreover, $[n]^*\omega=n^2\omega$.
 Thus, we get the following lemma. 

\begin{lem}
The Betti form $\omega$ is a real analytic, closed, semi-positive $(1,1)$-form on $\sA^o$ such that $\omega|_{\sA_b}=\omega_b$ 
for every point $b\in B^o$. In particular, the cohomology class of $\omega|_{\sA_b}$ coincides with $c_1(\sL|_{\sA_b})$ for
every $b\in B^o$.
\end{lem}

Since the monodromy of the foliation preserves the polarization $\sL_{\sA_b}$, it preserves $\omega_b$ and is contained in 
a symplectic group. 

\section{The canonical height and the Betti form}\label{sectioncanheight}

\subsection{The canonical height} \label{par:canonical height}
Recall that $K=\C(B)$. Let $X$ be any subvariety of $A_{\overline{K}}$. 
There exists a finite field extension $K'$ over $K$ such that $X$ is defined over $K'$; in other words, there exists a subvariety $X'$ of $A_{K'}$ such that $X=X'\otimes_{K'}\overline{K}.$
Let $\rho':B'\to B$ be the normalization of $B$ in $K'$. Set $\sA':=\sA\times_B B'$ and denote by $\rho: \sA'\to \sA$ the projection to the first factor; then, denote by $\sX'$ the Zariski closure of $X'$ in $\sA'.$
The {\bf{naive height}} of $X$ associated to the model $\pi\colon \sA\to B$ and the line bundles $\sL$ and $M$ is 
defined by the intersection number 
\begin{equation}
h(X)=\frac{1}{[K':K]}\left( \sX'\cdot c_1(\rho^*\sL)^{d_X+1}\cdot \rho^*\pi^*(c_1(M))^{d_B-1} \right)
\end{equation}
where $d_X=\dim X$ and $d_B=\dim B$. 
It depends on the model $\sA$ and the extension $\sL$ of $L$ to $\sA$ but it does not depend on the choice of~$K'$.

The {\bf{canonical height}} is the limit
\begin{equation}
\hat{h}(X)=\lim_{n\to +\infty} \frac{ h([n]_*X)}{n^{2(d_X+1)}}=\lim_{n\to +\infty} \frac{ \deg([n]|_X)h([n]X)}{n^{2(d_X+1)}}.
\end{equation}
It depends on $L$ but not on the model $(\sA, \sL)$; we refer to
Gubler's work \cite{Gubler:Pisa03} for more details. By \cite[Theorem 5.4, Chapter 6]{Lang1983}, the condition $\hat{h}(X)=0$ does not depend on $L$.
In particular, we may modify $\sL$ on special fibers to assume that {\sl{$\sL$ is ample}}.
See also \cite[Section 3]{Gubler2007}.

\medskip

Now we reformulate the canonical height in differential geometric terms.
For simplicity, assume that $X$ is already defined over $K$.
Set $\sA_1:=\sA$, $\pi_1:=\pi$ and $\sL_1:=\sL$.  Pick a K\"ahler form $\alpha_1$ in $c_1(\sL)$ (such a form exists because we choose $\sL$ ample).
For every $n\geq 1$, there exists an irreducible smooth projective scheme $\pi_n:\sA_n\to B$  over $B$, extending $\pi|_{\sA^o}:\sA^o\to B^o$, such that
the rational map $[n]:\sA^o\to \sA^o$ lifts to a morphism $f_n:\sA_n\to \sA$ over $B$.
Write $\sL_n:=f_n^*\sL$ and $\alpha_n:=f_n^*\alpha_1$.
Denote by $\sX_n$ the Zariski closure of $\sX^o$ in $\sA_n$.
Since the K\"ahler form $\nu$ introduced in Section~\ref{par:Notation} represents the class $c_1(M)$, the projection formula gives 
\begin{eqnarray}
\hat{h}(X) & = & \lim_{n\to \infty}n^{-2(d_X+1)}(\sX_n\cdot \sL_n^{d_X+1}\cdot (\pi_n^*M)^{d_B-1})\\
& = &\lim_{n\to \infty}n^{-2(d_X+1)}\int_{\sX_n}\alpha_n^{d_X+1}\wedge (\pi_n^*\nu)^{d_B-1}\\
& = &\lim_{n\to
  \infty}n^{-2(d_X+1)}\int_{\sX^{o}}([n]^*\alpha)^{d_X+1}\wedge
(\pi^*\nu)^{d_B-1}
\end{eqnarray}
because the integral on $\sX_n$ is equal to the integral on the dense Zariski open subset $\sX^o$ (and even 
on the regular locus $\sX^{o, reg}$).

\subsection{Gubler-Zhang inequality} 

By definition, the {\bf{essential height}} $\ess(X)$  of a subvariety $X\subset A$ is the real number
\begin{equation}
\ess(X)=\sup_Y\inf_{x\in X(\overline{K})\setminus Y}\hat{h}(x),
\end{equation}
 where $Y$ runs through all proper Zariski closed subsets of $X$.
The following inequality is due to Gubler in \cite[Lemma 4.1]{Gubler2007}; it is an analogue of Zhang's inequality \cite[Theorem 1.10]{Zhang1995} over number fields.
\begin{equation} \label{ineq:Gubler-Zhang}
0\leq \frac{\hat{h}(X)}{(d_X+1)\deg_L(X)}\leq \ess(X).
\end{equation}
The converse inequality $\ess(X)\leq \hat{h}(X)/\deg_L(X)$ also holds, but we shall not use it in this article. 

\begin{defi}We say that $X$ is {\bf{small}}, if $X_{\epsilon}$ is Zariski dense in $X$ for all $\epsilon>0$.
\end{defi}

The above inequalities comparing $\hat{h}(X)$ to $\ess(X)$ show that $X$ is small if, and only if $\hat{h}(X)=0.$

\begin{pro}\label{prosmall} Let $g:A\to A'$ be a morphism of abelian varieties over $K$, and let $a\in A(K)$ be a torsion point. Let $X$ be an absolutely irreducible subvariety of  $A$ over $K$.
\begin{enumerate}
\item If $X$ is small, then $g(X)$ is small.
\item If $g$ is an isogeny and $g(X)$ is small, then $X$ is small.
\item  $X$ is small if and only if $a+X$ is small.
\end{enumerate}
\end{pro}
\proof
Assertions (1) and (2) follow from \cite[Proposition 2.6.]{Yamaki2016}. 
To prove the third one fix an integer $n\geq 1$ such that $na=0$. 
By assertions (1) and (2),  $a+X$ is small if and only if $[n](a+X)=[n](X)$ is small,  if and only if $X$ is small.
\endproof

\subsection{Smallness and the Betti form} 

Here is the key relationship between the density of small points and the Betti form. 

\smallskip

\begin{thm-B} Let $X$ be an absolutely irreducible subvariety of  $A$ over $\C(B)$.
If $X$ is small,  then  
\[
\int_{\sX^{o}}\omega^{d_X+1}\wedge (\pi^*\nu)^{d_B-1}=0,
\]
with $\omega$ the Betti form associated to $L$ and $\nu$ the K\"ahler form on $B$ representing
the class $c_1(M)$.
\end{thm-B}

\begin{proof}

Since $X$ is small, $\hat{h}(X)=0$ and
equation (3.5) shows that 
\begin{equation}\label{eq:proofthmB}
0=\hat{h}(X)=\lim_{n\to
  \infty}n^{-2(d_X+1)}\int_{\sX^{o}}([n]^*\alpha)^{d_X+1}\wedge
(\pi^*\nu)^{d_B-1}.
\end{equation}
 
Let $U\subset  B^{o}$ be any relatively compact open subset of $B^o$ in the euclidean topology.
There exists a constant $C_U>0$ such that $C_U\alpha-\omega$ is semi-positive on $\pi^{-1}(U)$.
Since $[n]:\sA^o\to \sA^o$ is regular, the $(1,1)$-form $n^{-2}[n]^*(C_U\alpha-\omega)=C_Un^{-2}[n^*]\alpha-\omega$ is semi-positive.
Since $\omega$ and $\nu$ are semi-positive, we get
\begin{equation*}
0\leq \int_{\pi^{-1}(U)\cap\sX^{o}}\omega^{d_X+1}\wedge
(\pi^*\nu)^{d_B-1}
\leq \left(\frac{C_U}{n^2}\right)^{d_X+1}
\int_{\sX^{o}}([n]^*\alpha)^{d_X+1}\wedge (\pi^*\nu)^{d_B-1}
\end{equation*}
for all $n\geq 1$. Letting $n$ go to $+\infty$, equation~\eqref{eq:proofthmB} gives 
\begin{equation}
\int_{\pi^{-1}(U)\cap\sX^{o}}\omega^{d_X+1}\wedge (\pi^*\nu)^{d_B-1}=0.
\end{equation}
Since this holds for all relatively compact subsets $U$ of $B^{o}$, the theorem is proved. 
\end{proof}

\begin{cor}\label{corlocalintegral}Assume that $X$ is small.
Let $U$ and $V$ be open subsets of $B^o$ and $\sX^o$ with respect to the euclidean topology such that 
$U$ contains the closure of $\pi(V)$. 
Let $\mu$ be any  smooth real semi-positive $(1,1)$-form on $U$. We have 
\[
\int_{V}\omega^{d_X+1}\wedge (\pi^*\mu)^{d_B-1}=0.
\]
\end{cor}

\proof[Proof of the Corollary]
Since $\omega$ and $\mu$ are semi-positive, the integral is non-negative.
Since $\nu$ is strictly positive on $U$,  there is a constant $C>0$ such that $C\nu-\mu$ is semi-positive.
From Theorem~B we get
\begin{equation}
0\leq \int_{V}\omega^{d_X+1}\wedge (\pi^*\mu)^{d_B-1}\leq
C^{d_B-1}\int_{V}\omega^{d_X+1}\wedge (\pi^*\nu)^{d_B-1}=0,
\end{equation}
and the conclusion follows. \endproof

\begin{thm-BB} Assume that $X$ is small.
Then at every point $p\in \sX^o$, we have 
$T_p\sF\subseteq T_p{\sX^o}$. In other words, $\sX^o$ 
is invariant under the Betti foliation: for every  $p\in \sX^o$, 
the leaf $\sF_p$ is contained in $\sX^o$. 
\end{thm-BB}

\proof
We start with a simple remark. Let $P\colon \C^{N+1}\to \C^N$ be a complex linear map of rank $N$.
Let $\omega_0$ be a positive $(1,1)$-form on $\C^N$. If $V$ is a complex linear subspace of $\C^{N+1}$
of dimension $N$, then $\ker(P)\subset V$ if and only if $P\vert V$ is not onto, if and only if $(P^*\omega_0^N)\vert V =0$.
Now, assume that $B$ has  dimension $1$. Then, the integral of $\omega^{d_X+1}$ on $\sX^o$ vanishes; since
the form $\omega$ is non-negative, the remark implies that the kernel
of $P_p$ from Section~\ref{subsec:bettiform} is contained in $T_p\sX^o$ at every smooth point $p$
of~$\sX^o$. This proves the proposition when $d_B=1$. 

The general case reduces to $d_B=1$ as follows. 
Let $U$ and $U'$ be open subsets   of $B^o(\C)$ such that: (i) ${\overline{U}}\subset U'$ in the euclidean topology 
and (ii) there are  complex coordinates $(z_j)$ on $U'$ such that
$U=\{|z_j|<1, j=1,\dots,d_B\}$.
Set 
\begin{equation}
\mu:=i(dz_2\wedge d\overline{z_2}+\ldots +dz_{d_B}\wedge d\overline{z_{d_B}}).
\end{equation}
It is a smooth real non-negative $(1,1)$-form on $U'$. By Corollary \ref{corlocalintegral}, we have 
\begin{equation}
\int_{\pi^{-1}(U)\cap \sX}\omega^{d_X+1}\wedge (\pi^*\mu)^{d_B-1}=0.
\end{equation}
For $(w_2, \dots, w_{d_B})$ in $\C^{d_B-1}$ with norm $\vert w_i\vert<1$ for all $i$,  consider the 
slice 
\begin{equation}
\sX(w_2,\dots,w_{d_B})=\sX\cap\pi^{-1}(U\cap \{z_2=w_2,\dots,
z_{d_B}=w_{d_B}\});
\end{equation}
this slice provides a family of subsets of $\sA$ over the one-dimensional 
disk $\{(z_1,w_2, \ldots, w_{d_B})\; ; \; \vert z_1\vert <1\}$. Then, the integral of $\omega^{d_X+1}$
over $\sX(w_2, \dots, w_{d_B})$ vanishes for almost every point $(w_2,
\dots, w_{d_B})$; from the case $d_B=1$, 
we deduce that, at every point $p$ of $\sX^o\cap \pi^{-1}U$, the tangent $T_p\sX^o$ intersects $T_p\sF$ 
on a line whose projection in $T_{\pi(p)}B$ is the line $\{z_2=\cdots =z_{d_B}=0\}$. Doing the same for all coordinates $z_i$, 
we see that $T_p\sF$ is contained in $T_p\sX^o$.  \endproof

As a direct application of Theorem~B' and Remark \ref{remtrivial}, we prove Theorem~A in the isotrivival case.

\begin{cor}\label{cortripro} If $A_{\overline{K}}=A^{\overline{K}/\C}\otimes_{\C}\overline{K}$ and $X$ is small, then there exists a subvariety $Y\subseteq A^{\overline{K}/\C}$ such that $X\otimes_{K}\overline{K}=Y\otimes_{\C}\overline{K}.$
\end{cor}
\proof
Replacing $K$ by a suitable finite extension $K'$ and then  $B$ by its normalization in $K'$, we may assume that  $\sA^o=  B^o\times A^{\overline{K}/\C}$ and that $\pi\colon \sA^o\to B$ is the projection to the first factor. By Remark \ref{remtrivial}, the leaves of the Betti foliation are exactly the fibers of the projection $\pi_2$ onto the second factor.  Since $X$ is small, Theorem~B' shows that $\sX=\pi_2^{-1}(Y)$, with $Y:=\pi_2(\sX).$
\endproof
 
\section{Invariant analytic subsets of real and complex tori}\label{sectiondynainva}

Let $m$ be a positive integer. Let $M=\R^m/\Z^m$ be the torus of dimension $m$ and
$\pi\colon \R^m\to M$ be the natural projection.  The group $\GL_{m}(\Z)$ acts by real 
analytic homomorphisms on $M$. In this section, we study analytic subsets of $M$ which are invariant under the
action of  a subgroup $\Gamma\subset \SL_m(\Z)$. The main ingredient is a result of Muchnik 
and of Guivarc'h and Starkov. 

\subsection{Zariski closure of $\Gamma$} \label{par:notations-gamma-G}

We denote by
\begin{equation}
G={\mathrm{Zar}}(\Gamma)^{irr}
\end{equation}
 the neutral component, for 
 the Zariski topology, of the Zariski closure of $\Gamma$ in $\GL_m(\R)$.
 We shall assume that $G$ is semi-simple.
The real points $G(\R)$ form a real Lie group, and the neutral component in the euclidean topology 
is denoted $G(\R)^+$. 
Let $\Gamma_0$ be the intersection of $\Gamma$ with $G(\R)^+$; then $\Gamma_0$ is both contained in $\GL_m(\Z)$ and
Zariski dense in $G$: every polynomial equation that vanishes identically on $\Gamma_0$ vanishes also on $G$.
But the Zariski closure of $\Gamma_0$ in $\GL_m(\R)$ may be larger than $G(\R)^+$ (it may include other connected components). 

We shall denote by $V$ the vector space $\R^m$;  the lattice $\Z^m$ determines an integral, hence a rational structure on $V$.
The Zariski closures ${\mathrm{Zar}}(\Gamma)$ and ${\mathrm{Zar}}(\Gamma_0)$ are $\Q$-algebraic subgroups of $\SL_m$ for
this rational structure.

We shall say that $\Gamma$ (or $G$) has {\bf{no trivial factor}} if every $G$-invariant vector $u\in V$ is equal to $0$. Note that
this notion depends only on $G$, not on $\Gamma$.

\subsection{Results of Muchnik and Guivarc'h and Starkov}\label{par:Muchnik}
Assume that $V$ is an irreducible representation of $G$ over $\Q$; this means
that every proper $\Q$-subspace of $V$ which is $G$-invariant is the trivial subspace $\{0\}$.
We decompose $V$ into irreducible subrepresentations of $G$ over $\R$, 
\begin{equation}
V=W_1\oplus W_2\oplus \cdots \oplus W_s.
\end{equation}
To each $W_i$ corresponds a subgroup $G_i$ of $\GL(W_i)$ given by the restriction of the action of $G$ 
to $W_i$. Some of the groups $G_i(\R)$ may be compact, and we denote by $V_c$ the sum of the corresponding
subspaces: $V_c$ is the maximal $G$-invariant subspace of $V$ on which $G(\R)$ acts by a compact factor. 
It is a proper subspace of $V$; indeed, if $V_c$ were equal to $V$ then $G(\R)$ would be compact, $\Gamma$ would be finite, 
and $G$ would be trivial (contradicting the non-existence of trivial factor).

\begin{thm}[Muchnik \cite{Muchnik2005}; Guivarc'h and Starkov \cite{Guivarch2004}]\label{thm:Muchnik}
Assume that $G$ is semi-simple, and its representation on $\Q^m$ is irreducible. 
Let $x$ be an element of $M$. Then, one of the following two exclusive properties occur
\begin{enumerate}
\item the $\Gamma$-orbit of $x$ is dense in $M$; 
\item there exists a torsion point $a\in M$ such that $x\in a+\pi (V_c)$.
\end{enumerate}
\end{thm}

In the second assertion, the torsion point $a$ is uniquely determined by $x$, because otherwise  
 $V_c$ would contain a non-zero rational vector and the representation $V$ would not be irreducible over $\Q$. 
As a corollary, if $F\subset M$ is a closed, proper, connected and $\Gamma$-invariant subset, then  $F$
is contained in a translate of $\pi(V_c)$ by a (unique) torsion point. Also, if $x$ is a point of $M$ with a finite
orbit under the action of $\Gamma$, then $x$ is a torsion point.  
\begin{rem}
Theorem~\ref{thm:Muchnik} will be used to describe $\Gamma$-invariant
real analytic subsets $Z\subset M$. If it is infinite, such a set contains the image of a non-constant real analytic curve. 
The existence of such a curve in $Z$ is the main difficulty in Muchnik's argument, but in our situation it
is given for free. 
\end{rem}

\begin{rem}
Assume that $m=2g$ for some $g\geq 1$ and $M$ is in fact a complex torus $\C^g/\Lambda$, with $\Lambda\simeq \Z^{2g}$. 
Suppose that $F$ is a complex analytic subset of $M$. The inclusion $F\to M$ factors through
the Albanese torus $F\to A_F$ of $F$, via a morphism $A_F\to M$, and the image of $A_F$ is the quotient of a subspace $W$ in 
$\C^{g}$ by a lattice $W\cap \Lambda$. So, if $F\subset a+\pi(V_c)$, the subspace $V_c$ 
contains a subspace $W\subset \R^m$ which is defined over $\Q$, contradicting the irreducibility assumption. 
To separate clearly the arguments of complex geometry from the arguments of dynamical systems, we shall
not use this type of idea  before Section~\ref{par:Complex analytic invariant subsets}.
\end{rem}

\begin{rem}
Theorem~2 of \cite{Guivarch2004} should assume that the group $G$ has no compact factor (this is implicitely assumed
in \cite[Proposition~1.3]{Guivarch2004}). 
\end{rem}
  
\subsection{Invariant real analytic subsets} 
Let $F$ be an analytic subset of $M$. 
We say that $F$ does not {\bf{fully generate}} $M$ if there is a
proper subspace $W$ of $V$ and  a non-empty open subset $\mathcal U$ of $F$ such that  $T_xF\subset W$
for every regular point $x$ of $F$ in $\mathcal U$. Otherwise, we say that $F$ fully generates $M$.


\begin{pro}\label{pro:inv-analytic}
Let $\Gamma$ be a subgroup of $\GL_m(\Z)$. Assume that the neutral
component ${\mathrm{Zar}}(\Gamma)^{irr}\subset\GL_m(\R)$ is semi-simple, 
and has no trivial factor. Let $F$ be a real analytic and $\Gamma$-invariant subset of $M$. If $F$  
fully generates $M$, it is equal to $M$. 
\end{pro}

To prove this result, we decompose the linear representation of $G={\mathrm{Zar}}(\Gamma)^{irr}$ on $V$ into 
a direct sum of irreducible representations over $\Q$:
\begin{equation}
V=V_1\oplus \cdots \oplus V_s.
\end{equation}
 Since there is no trivial factor, non of the $V_i$ is the trivial representation. For each index $i$, we denote by $V_{i,c}$ the compact factor of $V_i$.
The projection $\pi$ is a diffeomorphism from $V_{i,c}$ onto its image in $M_i$, 
because otherwise $V_{i,c}$ would contain a non-zero vector in $\Z^m$ and $V_i$ would not be an irreducible representation over $\Q$. 
Set 
\begin{equation}
M_i=V_i/(\Z^m\cap V_i).
\end{equation}
Then, each $M_i$ is a compact torus of dimension $\dim(V_i)$, and $M$ is isogenous to the product of the $M_i$. 
We may, and we shall assume that $M$ is in fact equal to this product:
\begin{equation}
M=M_1\times \cdots \times M_s;
\end{equation}
this assumption simplifies the exposition without any loss of generality, because the image and the pre-image of a
real analytic set by an isogeny is analytic too.
We also assume, with no loss of generality, that $\Gamma$ is contained in $G$.
For every index $1\leq i\leq s$, we denote by $\pi_i$ the projection on the $i$-th factor $M_i$.

\begin{lem}
If $F$ fully generates $M$, the projection $F_i:=\pi_i(F)$ is equal to $M_i$ for every $1\leq i\leq s$.
\end{lem}

\begin{proof} By construction, $F_i$ is a closed, $\Gamma$-invariant subset of $M_i$. 
Fix a connected component $F_i^0$ of $F_i$. If it were contained in a translate of $\pi(V_{i,c})$, then 
$F$ would not fully generate $M$. Thus, Theorem~\ref{thm:Muchnik} implies  $F_i^0=M_i$. 
\end{proof} 

We do an induction on the number $s$ of irreducible factors. For just one factor, 
this is the previous lemma. 
Assuming that the proposition has been proven for $s-1$ irreducible factors, we now want to prove it for $s$ factors. 
To simplify the exposition, we suppose that $s=2$, which means that $M$ is the product of just two factors $M_1\times M_2$. 
The proof will only use that $\pi_1(f)= M_1$ and $F$ fully generates $M$; thus, 
changing $M_1$ into $M_1\times \ldots \times M_{s-1}$, this proof also establishes the induction in full generality. 

There is a closed 
subanalytic subset $Z_1$ of $M_1$ with empty interior 
such that $\pi_1$ restricts to a locally trivial analytic fibration from $F\setminus \pi_1^{-1}(Z_1)$ to $M_1\setminus Z_1$.  
If $F$ does not coincide with $M$, the fiber $F_{x}$ is a proper, non-empty analytic subset of $\{x\}\times M_2$ for 
every $x$ in $M_1\setminus Z_1$. We shall derive a contradiction 
from the fact that $F$ fully generates $M$.

Theorem~\ref{thm:Muchnik} tells us that, for every torsion point $x$ in $M_1\setminus Z_1$, there is a finite set of points $a_j(x)$ in $M_2$ such that 
\begin{equation}
F_{x}\subset \bigcup_{j=1}^J a_j(x)+\pi(V_{2,c});
\end{equation}
the number of such points $a_j(x)$ is bounded from above by the number of connected components of $F_x$.
Since torsion points are dense in $M_1$, this property holds for every point $x$ in $M_1\setminus Z_1$ (the $a_j(x)$
are not torsion points a priori).
Since there are points with a dense $\Gamma$-orbit in $M_1$, we can assume that the number $J$ of points $a_j(x)$ does not
depend on $x$. 

Assume temporarily that $J=1$, so that $F_x$ is contained in $a(x)+\pi(V_{2,c})$ for some point $a(x)$ of $M_2$. 
The point $a(x)$ is not uniquely defined by this property (one can replace it by $a(x)+\pi(v)$ for any $v\in V_{2,c}$), 
but there is a way to choose $a(x)$ canonically. 
First, the action of $G(\R)$ on $V_{2,c}$ factors through a compact subgroup of $\GL(V_{2,c})$, 
so we can fix a $G(\R)$-invariant euclidean metric  $\dist_2$ on $V_{2,c}$. 
Then, any compact subset $K$ of $V_{2,c}$ is contained in a unique ball of smallest radius 
for the metric $\dist_2$; we denote by $c(K)$ and $r(K)$ the center and radius of this ball. 
Since the projection $\pi$ is a diffeomorphism from $V_{2,c}$ onto its image in $M_2$,
 the center of $F_{x}$ inside the translate of $\pi(V_{2,c})$ containing $F_{x}$ is a well defined point 
\begin{equation}
c(x):=c(F_{x})
\end{equation}
of $M_2$ such that $F_{x}$ is contained in $c(x)+\pi(V_{2,c})$. When $J>1$, this procedure gives a finite set of   centers $\{c_j(x)\}_{1\leq j\leq J}$.

The centers $c_j(x)$ and the radii $r_j(x)$   are (restricted) sub-analytic functions of $x$. 
Thus, there is a proper, closed analytic subset $D_1$ of $M_1$, containing $Z_1$,
 such that all $r_j(x)$ and $c_j(x)$ are smooth and analytic on its complement (see~\cite{Bierstone-Milman:IHES, Coste, vandenDries:Book}). 
Let $\mathcal{G}$ be the subset of $\pi_1^{-1}(M_1\setminus D_1)$ given by the union of the graphs of the centers: $\mathcal{G}=\{(x,y)\in M_1\times M_2; \; 
x\in M_1\setminus D_1, \; y=c_j(x) \; {\text{for some}}\;  j\}$.

\begin{lem}
The set $\mathcal{G}$ is contained in finitely many translates of  subtori of $M_1\times M_2$, each of dimension $\dim M_1 $. 
\end{lem}

This lemma concludes the proof of Proposition~\ref{pro:inv-analytic}, because if $\mathcal{G}$ is locally contained in $a+\pi(W)$ for some proper
subset $W$ of $V$ of dimension $\dim M_1$, then $F$ is locally contained in $a+\pi(W+V_{2,c})$, and $F$ does not fully generate $M$ because
$\dim(W+V_{2,c})<\dim V$. 

\begin{proof}
By construction, $\mathcal{G}$  is a smooth analytic subset of $\pi_1^{-1}(M_1\setminus D_1)$ and it is invariant by $\Gamma$.
For $x$ in $M_1\setminus D_1$, we denote by $\mathcal{G}_x$ 
the finite fiber $\pi_1^{-1}(x)\cap \mathcal{G}$.

For every torsion point $x\in M_1\setminus D_1$, the stabilizer $\Gamma_{x}$ of $x$  
is a finite index subgroup of $\Gamma$ that preserves the finite set $\mathcal{G}_{x}$. Hence, $\mathcal{G}_{x}$ 
is a finite set of torsion points of $M$, and a finite index subgroup $\Gamma'_{x}$ of $\Gamma_x$  
fixes individually each of the points $z\in \mathcal{G}_{x}$. In particular, torsion points are dense in $\mathcal{G}$. 
Fix one of these torsion points $z=(x,y)$ with $x$ in $M_1\setminus D_1$, and consider the tangent subspace $T_z\mathcal{G}$. 
It is the graph of a linear morphism $\varphi_z\colon T_{x}M_1\to T_{y}M_2$. Identifying the tangent spaces $T_{x}M_1$ and $T_{y}M_2$
with $V_1$ and $V_2$ respectively, $\varphi_z$ becomes a morphism that interlaces the representations 
$\rho_1$ and $\rho_2$ of $\Gamma'_{x}$ on $V_1$ and $V_2$; since $\Gamma'_{x}$ is Zariski dense in 
$G$, we get
\begin{equation}\label{eqcon}
\rho_2(g)\circ \varphi_z=\varphi_z\circ \rho_1(g)
\end{equation}
for every $g$ in $G$. In other words, $\varphi_z\in \End(V_1;V_2)$ is a morphism of $G$-spaces.  
This holds for every torsion point $z$ of $\mathcal{G}$; by continuity of tangent spaces
and density of torsion points, this holds everywhere on $\mathcal{G}$. 

Since $\mathcal{G}$ is $\Gamma$-invariant, we also have
\begin{equation}
\varphi_{g(z)}\circ \rho_1(g)= \rho_2(g)\circ\varphi_{z}
\end{equation}
for all $g\in \Gamma$ and $z\in \mathcal{G}$.
Then equation \eqref{eqcon} shows that $\varphi_{g(z)}=\varphi_{z}$, which means that 
 the tangent space $T_z\mathcal{G}$ is constant along the orbits of $\Gamma$.
Taking a point $z$ in $\mathcal{G}$ whose first projection has a dense $\Gamma$-orbit in $M_1$, we see that
the tangent space $w\in \mathcal{G}\mapsto T_w\mathcal{G}$ takes only finitely many values, at most $\vert \mathcal{G}_{\pi_1(z)}\vert$. 

Let $(W_j)_{1\leq j\leq k}$ be the list of possible tangent spaces $T_z\mathcal{G}$. Locally, near any point $z\in \mathcal{G}$, $\mathcal{G}$  coincides with $z+\pi(W_j)$ for some $j$. By analytic continuation $\mathcal{G}$ contains the intersection of $z+\pi(W_j)$ with $\pi_1^{-1}(M_1\setminus D_1)$; thus, $W_j$ is a rational subspace of $V$ and $\pi(W_j)$ is a  subtorus of $M$. Then $\mathcal{G}$ is  contained in a finite union of translates of the tori $\pi(W_j)$.
\end{proof}

\subsection{Complex analytic invariant subsets}\label{par:Complex analytic invariant subsets}

Let $\J$ be a complex structure on $V=\R^m$, so that $M$ is now endowed with 
a structure of complex torus. Then, $m=2g$ for some integer $g$, $\R^m$ can be identified to $\C^g$, 
and $M=\C^g/\Lambda$ where $\Lambda$ is the lattice $\Z^m$; to simplify the exposition, we denote by $A$ the complex
torus $\C^g/\Lambda$ and by $M$ the real torus $\R^m/\Z^m$. Thus, $A$ is just $M$, together with the complex 
structure $\J$.
Let $X$ be an irreducible complex analytic subset of $A$, and let $X^{reg}$ be its smooth locus.

\begin{lem}\label{lemminimaltorus} Let $W$ be the real subspace of $V$ generated by  the tangent 
spaces $T_xX$, for $x\in X^{reg}$. Then $W$ is both complex and rational, and $X$ is contained in 
a translate of the complex torus $\pi(W)$. 
\end{lem}

\proof
Since $X$ is complex, its tangent space is invariant under the complex structure:
$\J T_xX=T_xX$ for all $x\in X^{reg}$. So, the sum $W:=\sum_x T_xX$ of the $T_xX$ over all points  $x\in X^{reg}$ is invariant by $\J$
and $W$ is a complex subspace of $V\simeq \C^g$.
Observe that if $V'$ is any real subspace of $V$ such that $\pi(V')$ contains some translate of $X^{reg}$, then $W\subseteq V'$.

Let $a$ be a point 
of $X^{reg}$, and $Y$ be the translate $X-a$ of $X$. It is an irreducible complex analytic subset
of $A$ that contains the origin $0$ of $A$ and satisfies $T_yY \subset W$ for every $y\in Y^{reg}$. Thus, 
$Y^{reg}$ is contained in the projection $\pi(W)\subset A$. Set $Y^{(1)}=Y$, $Y^{(1)}_o=Y^{reg}$ and then 
\begin{equation}
Y^{(\ell+1)}=Y^{(\ell)}-Y^{(\ell)}, \quad Y^{(\ell+1)}_o=Y^{(\ell)}_o-Y^{(\ell)}_o
\end{equation}
for every integer $\ell \geq 1$. Since $Y^{(1)}$ is irreducible, and $Y^{(2)}$ is the image of $Y^{(1)}\times Y^{(1)}$
by the complex analytic map 
$(y_1,y_2)\mapsto y_1-y_2,
$ we see that $Y^{(2)}$ is an irreducible complex analytic subset of $A$. Moreover $Y^{(2)}_o$ is a connected, dense, and open subset of $Y^{(2),reg}$.
Observe that $Y^{(2)}_o$ is contained in $\pi(W)$ and   contains 
$Y^{(1)}_o$ because $0\in Y^{(1)}_o$. By a simple induction, the sets $Y^{(\ell)}$
form an increasing sequence of irreducible complex analytic subsets of $A$, and  $Y^{(\ell)}_o$ is a connected, dense and open subset of $Y^{(\ell), reg}$ that is contained in $\pi(W)$. 
By the Noether property, there is an index $\ell_0\geq 1$ such that $Y^{(\ell)}=Y^{(\ell_0)}$ for every $\ell\geq \ell_0$. This 
complex analytic set is a subgroup of $A$, hence it is a complex subtorus. Write $Y^{(\ell_0)}=\pi(V')$ for some rational subspace $V'$ of $V$. Since $Y\subset \pi(V')$, we get $W\subseteq V'$.
Since $Y^{(\ell_0)}_o\subseteq \pi(W)$, we derive $V'=T_xY^{(\ell_0)}_o\subseteq W$ for every $x\in Y^{(\ell_0)}_o.$ This implies $W=V'$, and shows that $W$ is rational.

Thus, $\pi(W)$ is a complex subtorus of $A$. Since $T_xX$ is contained in $W$ for every regular point, $X$ is locally contained in a translate of $\pi(W)$. 
Being irreducible, $X$ is connected, and it is contained in a unique translate $a+\pi(W)$. 
\endproof

\begin{lem}\label{lem:generates} Let $X$ be an irreducible complex analytic subset of $A$. The following properties are equivalent: 
\begin{itemize}
\item[(i)] $X$ is   contained in a translate of a proper complex subtorus $B\subset A$;
\item[(ii)] $X$ does not fully generate $M$;
\item[(iii)] there is a proper real subspace $V'$ of $V$ that contains $T_xX$ for every $x\in X^{reg}$.   
\end{itemize}
\end{lem}

\begin{proof}
Obviously (i) $\Rightarrow$  (iii) $\Rightarrow$ (ii). We now prove that (ii) implies~(i).
If $X$ does not fully generate $M$, then (iii) is satisfied on some non-empty open subset $\mathcal U$ of $X^{reg}$. 
Since $X^{reg}$ is connected and locally analytic, we deduce from analytic continuation that $T_xX\subset V'$ for every regular point of $X$. 
From Lemma~\ref{lemminimaltorus}, $X$ is contained in a complex subtorus $B=\pi(W)\subset A$ for some complex subspace $W$ of $V'$.
\end{proof}

\begin{thm}\label{thminvcomplexsub}
  Let $\Gamma$ be a subgroup of $\SL_m(\Z)$.
Assume that the neutral component for 
 the Zariski topology of the Zariski closure of $\Gamma$ in $\SL_m(\R)$
  is semi-simple
and has no trivial factor.
Let $\J$ be a complex structure on $M=\R^m/\Z^m$ and let $X$ be an irreducible complex analytic subset of the complex
torus $A=(M,\J)$. If $X$ is $\Gamma$-invariant, it is equal to a translate of a complex subtorus $B\subset A$ by a torsion point. 
\end{thm}

\begin{proof}
Set $W:=\sum_{x\in X^{reg}}T_xX$.  Lemma \ref{lemminimaltorus} shows that $W$ is complex and rational. Since $X$ is $\Gamma$-invariant, so is $W$.
Its projection $B =\pi(W)$ is a complex subtorus of $A$ such that 
\begin{enumerate}
\item $B$ is $\Gamma$-invariant;
\item $B$ contains a translate $Y=X-a$ of $X$;
\item $Y$ fully generates $B$.
\end{enumerate}
The group $\Gamma$
acts on the quotient torus $A/B$ and preserves the image of $X$, {\sl i.e.} the image $\overline{a}$ of $a$.
Since $V$ has no trivial factor, ${\overline{a}}$ is a torsion 
point of $A/B$. Then there exists a torsion point $a'$ in $A$ such that $X\subseteq a'+B$. 
Replacing $a$ by $a'$ and $\Gamma$ by a finite index subgroup $\Gamma'$ which fixes $a'$, 
we may assume that $a$ is torsion and $Y=X-a$ is invariant by $\Gamma$.
We apply Proposition~\ref{pro:inv-analytic}
to $B$, the restriction $\Gamma_B$ of $\Gamma$ to $B$, and the complex analytic subset $Y$: we conclude via Lemma~\ref{lem:generates} that
$Y$ coincides with $B$. Thus, $X=a+B$. 
\end{proof}

%
%

\section{Proof of Theorem~A}\label{sectionproof}

By base change, we may suppose that $X$ is an absolutely irreducible subvariety of $A$. 
We assume that $X$ is small ($X_\epsilon$ is dense in $X$ for all $\epsilon >0$),  
and prove that  $X$ is a torsion coset of $A$.

\subsection{Monodromy and invariance}\label{par:modo-inv-Deligne}

Let $b\in B^o$ be any point.  The monodromy $\rho: \pi_1(B^o)\to \GL_{2g}(\Z)$ of the Betti foliation maps 
the fundamental group of $\pi_1(B^o)$ onto a subgroup  $\Gamma:=Im(\rho)$ of  $\GL_{2g}(\Z)$ that acts by linear diffeomorphisms on 
the torus $\sA_b\simeq \R^{2g}/\Z^{2g}$. As in Section~\ref{par:notations-gamma-G}, we denote by $G$ the neutral component $Zar(\Gamma)^{irr}$. 
We let $V^{G}$ denote the subspace of elements $v\in \R^{2g}$ which are fixed by $G$.
By Deligne's semi-simplicity theorem, the
group $G$ is semi-simple (see \cite[Corollary 4.2.9]{Deligne1971}).
Theorem~B' implies that $X$ is invariant under the Betti foliation, so that $X_b$ is invariant under the action of $\Gamma$. 

\subsection{Trivial trace}\label{sectri}

We first treat the case when $A^{\overline{K}/\C}$ is trivial. 
According to \cite[Theorem 1.5]{Yamaki2016a}, this is the only case we need to treat. However we shall also treat the case of 
a non-trivial trace below for completeness. 

By \cite[Corollary 4.1.2]{Deligne1971} and \cite{Grothendieck1966} (see also \cite[4.1.3.2]{Deligne1971}), we have $V^{G}=\{0\}$
and  Theorem~\ref{thminvcomplexsub} implies that $\sX_b$ is a translation of an abelian subvariety of $\sA_b$ by some
 torsion point  $y_b\in \sA_b$. 
Observe that the leaf $\sF_{y_b}$ is an algebraic muti-section of $\sA^o$ 
(see Remark~\ref{rem:torsion-algebraic-leaves}). 
By base change, we may assume that $\sF_{y_b}$ is a section and is the Zariski closure of a torsion point $y\in A(K)$ in $\sA^o$. 
Theorem~B' shows that $y\in X$, and replacing $X$ by $X-y$ we may suppose that $0\in X$;  then $\sX_b$ is an abelian subvariety of $\sA_b$ for all $b\in B^o$. It follows that $\sX^o$ is a subscheme of the abelian scheme $\sA^o$ over $B^o$ which is stable under the group laws. So $X$ is an abelian subvariety of $A$.

\subsection{The general case}

We do not assume anymore that $A^{\overline{K}/\C}$ is trivial. Set $A^t=A^{\overline{K}/\C}\otimes_{\C}K$. 
Replacing $K$ by a finite extension and $A$ by a finite cover, we assume that $A=A^t\times A^{nt}$ where $A^{nt}$ is an abelian variety over $K$
with trivial trace. We also choose the model $\sA$ so that $\sA^o=(\sA^t)^o\times_{B^o} (\sA^{nt})^o$ where $(\sA^t)^o$ and $(\sA^{nt})^o$ are the Zariski closures of $A^{t}$ and $A^{nt}$ in $\sA^o$ respectively.  Denote by $\pi^t:\sA^o\to (\sA^t)^o$ the projection to the first factor and $\pi^{nt}:\sA^o\to (\sA^{nt})^o$ the projection to the second factor.
After replacing $K$ by a further finite extension and   $B$ by its normalization, we may assume that $(\sA^t)^o=A^{\overline{K}/\C}\times B^o$. 
 Note that  $\pi^t|_{\sA^{t}_b}: \sA^{t}_b\to A^{\overline{K}/\C}$ is an isomorphism
for every fiber $\sA^t_b$ with $b\in B^o$. 

By Proposition~\ref{prosmall}-(i), the generic fibers of $\pi^t(\sX^o)$ and $\pi^{nt}(\sX^o)$ are small. Corollary \ref{cortripro} shows that $\pi^t(\sX^o)=Y\times B^o$ for some subvariety $Y$ of $A^{\overline{K}/\C}$. Section~\ref{sectri} shows that the geometric generic fiber of $\pi^{nt}(\sX^o)$ is
a torsion coset $a+\sA'$ for some torsion point $ a\in A^{nt}_{\overline{K}}(\overline{K})$ and some abelian subvariety $A'$.  Replacing $K$ 
by a finite extension, we may assume that $a$ and $A'$ are defined over $K$.
We have that $\sX^{o}\subseteq \pi^t(\sX)\times_{B^o}\pi^{nt}(\sX)=\pi^t(\sX)+\pi^{nt}(\sX)$ and we only need to show that $\sX^{o}= \pi^t(\sX)\times_{B^o}\pi^{nt}(\sX).$


For every $b\in B^o,$ $\sA_b=\sA^t_b\times \sA^{nt}_b.$ The monodromy on $\sA_b$ is the diagonal product of the monodromies on each factor. It is trivial on the first one so, for every $x\in \sA^t_b$, the fiber $\pi^t|_{\sA_b}^{-1}(x)\simeq \sA^{nt}_b$ is invariant under $\Gamma$. 
It follows that $\pi^t|_{\sA_b}^{-1}(x)\cap \sX_b$ is also $\Gamma$-invariant. By Theorem \ref{thminvcomplexsub}, $\pi^{nt}(\pi^t|_{\sA_b}^{-1}(x)\cap \sX_b)\subseteq \pi^{nt}(\sX_b)$ is a torsion coset of the abelian variety $\sA^{nt}_b$. Since the set of all torsion cosets of $\pi^{nt}(\sX_b)$ is countable, $\pi^{nt}(\pi^t|_{\sA_b}^{-1}(x) \cap \sX_b)$ does not depend on $x\in \pi^t(\sX_b)$. Hence, $\sX_b=\pi^t(\sX_b)\times \pi^{nt}(\sX_b)$ for all $b\in B^o.$ Then $\sX^{o}= \pi^t(\sX)\times_{B^o}\pi^{nt}(\sX)$ which concludes the proof.
\endproof

%
%

\bibliographystyle{plain}
\bibliography{gbc}

\end{document}